\documentclass{birkjour}

\usepackage{color}
\usepackage{xcolor}
\usepackage{amssymb,amsmath,mathtools}
\usepackage{amsbsy}
\usepackage{mathrsfs}
\usepackage{color}
\usepackage[active]{srcltx}
\usepackage{bm}
\usepackage{fixmath}
\usepackage{enumitem}
\usepackage{hyperref}

\newcommand{\capac}{\mathrm{Cap}}

\newcommand{\supp}{\mathop{\mathrm{supp}}}\usepackage{color}

 \newtheorem{thm}{Theorem}[section]
 \newtheorem{cor}[thm]{Corollary}
 \newtheorem{lem}[thm]{Lemma}
 
 \theoremstyle{definition}
 
 \theoremstyle{remark}

 \numberwithin{equation}{section}

\begin{document}

%
%
%
%
%
%
%
%
%

\title[]
 {Sufficient criteria for obtaining Hardy inequalities on Finsler manifolds}

\author{\'{A}gnes Mester}

\address{%
Institute of Applied Mathematics \\ 
\'Obuda University \\ 
B\'ecsi \'ut 96B \\ 
1034 Budapest, Hungary.}

\email{mester.agnes@stud.uni-obuda.hu}

\author{Ioan Radu Peter}
\address{Department  of Mathematics \\ 
	Technical University of Cluj-Napoca \\ 
	Str. Memorandumului nr. 28 \\ 
	400114 Cluj-Napoca, Romania.}

\email{ioan.radu.peter@math.utcluj.ro}

\author{Csaba Varga}
\address{Faculty of Mathematics and Computer Science \\
	Babe\c{s}-Bolyai University \\ 
	Str. Mihail Kogălniceanu nr. 1 \\
	400084 Cluj-Napoca, Romania \\ 
	\& Department of Mathematics \\ 
	University of P\'{e}cs \\ 
	Ifj\'{u}s\'{a}g \'{u}tja 6\\ 
	7624 P\'{e}cs, Hungary.}

\email{csvarga@cs.ubbcluj.ro}



\begin{abstract}
	We establish Hardy inequalities involving a weight function on complete, not necessarily reversible Finsler manifolds.
	We prove that the superharmonicity of the weight function provides a sufficient condition to obtain Hardy inequalities.
	Namely, if $\rho$ is a nonnegative function and $-\boldsymbol{\Delta} \rho \geq 0$ in weak sense, where 
	$\boldsymbol{\Delta}$ is the Finsler-Laplace operator defined by 
	$ \boldsymbol{\Delta} \rho = \mathrm{div}(\boldsymbol{\nabla} \rho)$, then we obtain the generalization of some Riemannian Hardy inequalities given in D'Ambrosio and Dipierro \cite{DAmbrosio}.
	
	By extending the results obtained, we prove a weighted Caccioppo\-li-type inequality, a Gagliardo-Nirenberg inequality and a Heisenberg-Pauli-Weyl uncertainty principle on complete Finsler ma\-nifolds. 
	Furthermore, we present some Hardy inequalities on Finsler-Hadamard manifolds with finite reversibility constant, 
	by defining the weight function with the help of the distance function.  
	Finally, we extend a weighted Hardy-inequality to a class of Finsler manifolds of bounded geometry.

\end{abstract}

\maketitle

\section{Introduction}

Consider the Euclidean space $\mathbb{R}^n$ with $n \geq 2$, and let $p \in (1,n)$. 
The classical multi-dimensional Hardy inequality asserts that
$$ \bigg( \frac{n-p}{p} \bigg)^p  \int_{\mathbb{R}^n} \frac{|u|^p}{|x|^p} dx \leq \int_{\mathbb{R}^n} |\nabla u|^p dx , 
\quad \forall u \in C_0^{\infty}(\mathbb{R}^n) ,$$
where the constant $ \big( \frac{n-p}{p} \big)^p$ is sharp, see for instance Balinsky, Evans and Lewis \cite[Section 1.2]{BalEvLew1}. This result has been generalized and improved in several directions, considering weighted Hardy inequalities, adding remainder terms, analyzing the best constant and the existence of minimizers, or replacing the set $\mathbb{R}^n$ by a (bounded or convex) domain $\Omega \subset \mathbb{R}^n$. For more details see, for example, Barbatis, Filippas and Tertikas \cite{BFT}, Brezis and Marcus \cite{BrezMarc}, Brezis and V\'azquez \cite{BV},  D'Ambrosio \cite{DAmbrosio2}, Gazzola, Grunau and Mitidieri \cite{GGM}, Lewis, Li and Li \cite{LewisLi} and references therein. 

In the last 20 years there has been a growing effort to extend these Hardy inequalities to Riemannian manifolds. In 1997 Carron \cite{Carron} established a method to obtain weighted $L^2$-type Hardy inequalities on complete non-compact Riemannian manifolds. This was followed by the generalizations and improvements obtained by Kombe and \"{O}zaydin \cite{Kombe1,Kombe2}, D'Ambrosio and Dipierro \cite{DAmbrosio}, Yang, Su and Kong \cite{YSK} and Xia \cite{Xia}. 

Moreover, recent advancements were made in the study of Hardy and Rellich inequa\-lities on Finsler manifolds. For instance, Krist\'aly and Repov\v{s} \cite{KristalyRepovs} considered Hardy and Rellich inequalities on reversible Finsler-Hadamard manifolds, which was followed by the paper of Yuan, Zhao and Shen \cite{YZS}, improving these inequalities on non-reversible Finsler manifolds. Recently, Bal \cite{Bal} and  Mercaldo, Sano and Takahashi \cite{MST} studied anisotropic Hardy inequalities for the Finsler $p$-Laplacian on reversible Minkowski spaces, then Zhao \cite{Zhao} proved weighted $L^p$-Hardy inequalities on non-reversible Finsler manifolds. Note that these investigations applied constraints on the mean covariation and the flag curvature, and showed that the results obtained depend deeply on the curvature of the manifold and on the non-Riemannian nature of the Finsler structure, measured by the reversibility constant and uniformity constant (see Section \ref{preliminaries}).

In this paper we establish a method to obtain weighted Hardy inequalities on forward complete, not necessarily reversible Finsler manifolds, extending the method given by D'Ambrosio and Dipierro \cite{DAmbrosio} and complementing some of the results obtained by Zhao \cite{Zhao}. 
We prove that the superharmonicity of the weight function is sufficient to obtain Hardy inequalities in the Finslerian setting, 
without applying further assumptions on the geometric properties of the manifold. 
In order to avoid further technicalities, we consider only the case $p = 2$, therefore obtaining $L^2$-type Hardy inequalities.
However, by applying suitable modifications in the proofs, our results can be extended to any $p > 1$. Also, note that the inequalities obtained can be established on backward complete Finsler manifolds in a similar manner.

In order to present our main results, we briefly introduce some notations, for the detailed definitions see Section \ref{preliminaries}. 
Let $(M,F,\mathsf{m})$ be a forward complete Finsler manifold endowed with a not necessarily reversible Finsler structure $F$, the polar transform $F^{*}$ and a smooth measure $\mathsf{m}$, and let $\Omega \subset M$ be an open set. Suppose that $\rho \in W_{\mathrm{loc}}^{1,2}(\Omega)$ is a nonnegative weight function such that $\rho$ is superharmonic on $\Omega$ in weak sense, i.e. 
$- \boldsymbol{\Delta} \rho \geq 0$ on $\Omega$ in the distributional sense. 
Here   
$\boldsymbol{\Delta} \rho = \mathrm{div}(\boldsymbol{\nabla} \rho)$, and 
$\boldsymbol{\Delta}, \mathrm{div}, \boldsymbol{\nabla}$ denote the Finsler-Laplace operator, the divergence and the 
gradient \mbox{operator}, respectively. Then we have the following weighted $L^2$-Hardy inequality: 

\begin{thm} \label{Hardy1}
	Let $\rho \in W_{\mathrm{loc}}^{1,2}(\Omega)$ be a nonnegative function such that $\rho$ is superharmonic on $\Omega$ in weak sense.
	Then $\frac{F^{*2}(D\rho)}{\rho^2} \in L_{\mathrm{loc}}^1(\Omega)$ and the Hardy inequality 
	\begin{equation}\label{ineqmain}
	\int_{\Omega} \frac{u^2}{\rho^2} F^{*2}(x, D\rho) ~\mathrm{d}\mathsf{m}
	\leq 4 \int_{\Omega}F^{2}(x, \boldsymbol{\nabla}u) ~\mathrm{d}\mathsf{m}
	\end{equation}
	holds for every function $u \in C^{\infty}_0(\Omega)$. 
\end{thm}

The proof is based on the divergence theorem, see Ohta and Sturm \cite{Ohta-Sturm}. 

By further developing this technique, we obtain the following weighted Gagliardo-Nirenberg inequality:
\begin{cor} \label{GN}
	Let $\rho \in W_{\mathrm{loc}}^{1,2}(\Omega)$ be a nonnegative function such that $\rho$ is superharmonic on $\Omega$ in weak sense. Then
	\begin{equation*} 
	\int_{\Omega} u^2 \frac{F^*(x, D \rho)}{\rho} ~\mathrm{d}\mathsf{m}
	\leq 2 \left( \int_{\Omega} u^{2} ~\mathrm{d}\mathsf{m}\right)^{\frac{1}{2}}
	\left( \int_{\Omega} F^2(x, \boldsymbol{\nabla}u) ~\mathrm{d}\mathsf{m}\right)^{\frac{1}{2}}, 
	~ \forall u \in C^\infty_0(\Omega). 
	\end{equation*}	 
\end{cor}

Similarly, we also prove the following weighted Heisenberg-Pauli-Weyl uncertainty principle:
\begin{cor} \label{HPW}
	Let $\rho \in W_{\mathrm{loc}}^{1,2}(\Omega)$ be a nonnegative function such that $\rho$ is superharmonic on $\Omega$ in weak sense. Then
	\begin{equation*} 
	\int_{\Omega} u^2 ~ \mathrm{d}\mathsf{m}
	\leq 2 \left( \int_{\Omega} F^2(x, \boldsymbol{\nabla}u) ~\mathrm{d}\mathsf{m}\right)^{\frac{1}{2}}
	\left( \int_{\Omega} u^{2} \frac{\rho^{2}}{F^{*2}(x, D \rho)} ~\mathrm{d}\mathsf{m}\right)^{\frac{1}{2}}, 
	\forall u \in C^\infty_0(\Omega). 
	\end{equation*}	
\end{cor} 

Note that when $(M,F) = (M,g)$ is a Riemannian manifold, the Finsler-Laplace ope\-rator $\boldsymbol{\Delta}$ and the gradient $\boldsymbol{\nabla}$ reduce to the Laplace-Beltrami operator $\Delta_g$ and the usual gradient operator $\nabla$. 
Furthermore, by the Riesz representation theorem, one can identify the tangent space $T_xM$ and its dual space $T_x^*M$, and the Finsler metrics $F$ and $F^*$ become the norm $|\cdot|_g$ induced by the Riemannian metric $g$. Therefore, our results extend the Hardy inequalities obtained in D'Ambrosio and Dipierro \cite{DAmbrosio} to the class of forward complete Finsler manifolds.

Finally, we extend Theorem \ref{Hardy1} to the  
class of Finsler manifolds of bounded geometry in the sense of Ohta and Sturm \cite{Ohta-Sturm}, as follows. 
\begin{cor} \label{corollForInclusion2}
	Let $(M,F)$ be a geodesically complete, $2$-hyperbolic Finsler ma\-nifold such that $F$ is uniformly convex, $r_F < \infty$ and 
	$\mathrm{Ric}_N \geq -\kappa$, $\kappa > 0$. 
	Let $\Omega \subset M$ be an open set such that $M \setminus \Omega$ is compact with $\capac_2(M \setminus \Omega, M) = 0$. 
	If $\rho \in W_{\mathrm{loc}}^{1,2}(\Omega)$ is a nonnegative function such that $\rho$ is superharmonic on $\Omega$ in weak sense, then the Hardy inequality 
	\begin{equation*}
	\int_{\Omega} \frac{u^2}{\rho^2} F^{*2}(x, D\rho) ~\mathrm{d}\mathsf{m}
	\leq 4 \int_{\Omega}F^{2}(x, \boldsymbol{\nabla}u) ~\mathrm{d}\mathsf{m}
	\end{equation*}
	holds for every function $u \in C^{\infty}_0(M)$.
\end{cor}
Here $r_F$ and $\mathrm{Ric}_N$ denote the reversibility constant and  weighted Ricci curvature, res\-pectively, while $(M,F)$ is said to be $2$-hyperbolic if there exists a compact set $K \subset M$ with non-empty interior such that $\capac_2 (K,M) > 0$ (for further details see Sections \ref{preliminaries} and \ref{extension}).

The paper is organized as follows. Section \ref{preliminaries} recalls some definitions and results from Finsler geometry.
In Section \ref{hardy} we prove Theorem \ref{Hardy1}, an extension of this result and a Caccioppoli-type inequality. 
Then, as application, we present some Hardy-type \mbox{inequalities} involving the Finslerian distance function on Finsler-Hadamard manifolds.  
Section \ref{generalizations} presents a weighted Gagliardo-Nirenberg inequality and a Heisenberg-Pauli-Weyl uncertainty principle, which imply - as particular cases - Corollaries \ref{GN} and \ref{HPW}, respectively.
Finally, in Section \ref{extension} we introduce the notion of $2$-capacity, and we extend the Hardy inequality \eqref{ineqmain} 
to the class of functions $C^{\infty}_0(M)$ when $(M,F)$ is a $2$-hyperbolic manifold of bounded geometry.

\section{Preliminaries on Finsler geometry}    \label{preliminaries}

\subsection{Finsler manifolds}

Let $M$ be a connected $n$-dimensional smooth manifold and $TM=\bigcup_{x \in M}T_{x} M $ its tangent bundle, where $T_{x} M$ denotes the tangent space of $M$ at the point $x$.

A continuous function $F: TM \to [0,\infty)$ is called a Finsler structure on $M$ if the following conditions hold:
\begin{enumerate}[label=(\roman*)]
	\item $F$ is $C^{\infty}$ on $TM\setminus\{ 0 \}$ ; 
	\item $F(x,\lambda y) = \lambda F(x,y)$ for all $\lambda \geq 0$ and all $(x,y)\in TM$ ;
	\item the $n \times n$ Hessian matrix $\Big( g_{ij}(x,y) \Big) := \left( \left[ \frac{1}{2}F^{2}(x,y)\right] _{y^{i}y^{j}} \right)$
	is positive definite for all $(x,y)\in TM\setminus \{0\}.$
\end{enumerate}
Such a pair $(M,F)$ is called a Finsler manifold.

If, in addition, $F(x,\lambda y) = |\lambda| F(x,y)$ holds for all $\lambda \in \mathbb{R}$ and all $(x,y) \in TM$, then the 
Finsler manifold is called reversible. Otherwise, $(M,F)$ is non-reversible.

Now let $(M,F)$ be a connected $n$-dimensional $C^{\infty}$ Finsler manifold. 

The reversibility constant of $(M,F)$ is defined by the number
\begin{equation} \label{reversibility_constant}
r_{F} = \sup_{x\in M} ~ \sup_{\substack{ y\in T_{x} M \setminus \{0\}}} \frac{F(x,y)}{F(x,-y)},
\end{equation}
and it measures how much the manifold deviates from being reversible (see Rademacher \cite{Rademacher}).  
Note that $r_F \in [1, \infty]$ and $r_F = 1$ if and only if $(M,F)$ is reversible. 
Similarly, one can define the constant $r_{F^*}$ by means of the polar transform $F^*$ (see Section \ref{polar}), 
and one can prove that $r_F = r_{F^*}$.

Unlike the Riemannian metric, the Finsler structure does not induce a unique natural connection on the Finsler manifold $(M, F)$.
However, it is possible to define on the pull-back tangent bundle $\pi^* TM$ a linear, torsion-free and almost metric-compatible connection called the Chern connection, see Bao, Chern and Shen \cite[Chapter~2]{BCS}.   

With the help of the Chern connection, one can define the Chern curvature tensor $\textbf{R}$ and the flag curvature $\textbf{K}$, see Bao, Chern and Shen \cite[Chapter~3]{BCS}. 
For a fixed point $x \in M$ let $y,v \in T_x M$ be two linearly independent tangent vectors.
Then the flag curvature is defined as
$$ \textbf{K}^y(y, v) = \frac{g_y(\textbf{R}(y,v)v,y)}{ g_y(y,y) g_y(v,v) - g_y(y,v)^2} ,$$ 
where $g$ is the fundamental tensor on $\pi^* TM$ induced by the Hessian matrices $(g_{ij})$. 
We say that $(M,F)$ has non-positive flag curvature if $\textbf{K}^y(y,v) \leq 0$ for every $x \in M$ and every choice of $y,v \in T_x M$, and we denote it by $\textbf{K} \leq 0$. The Ricci curvature is defined as
$$ \mathrm{Ric}(y) = F^2(y) \sum_{i=1}^{n-1} \textbf{K}^y(y, e_i),$$
where $\{ e_1, \cdots ,e_{n-1}, \frac{y}{F(y)} \}$ is an orthonormal basis of $T_xM$ with respect to $g_y$.

The Chern connection also induces the notion of covariant derivative and  parallelism of a vector field along a curve. A curve $\gamma: [a,b] \to M$ is called a geodesic if its velocity field $\dot \gamma$ is parallel along the curve, i.e. $D_{\dot \gamma} \dot \gamma = 0$.

$(M,F)$ is said to be forward (or backward, respectively) complete if every geodesic $\gamma: [a,b] \to M$ can be extended to a geodesic defined on $[a, \infty)$ (or $(-\infty, b]$, respectively). In particular, a Finsler manifold is said to be complete if it is forward and backward complete. 
Note that if the reversibility constant $r_F < \infty$, then the forward and backward completeness of $(M,F)$ are equivalent (the proof is based on the Hopf-Rinow theorem, for further details see Bao, Chern and Shen \cite[Section~6.6]{BCS}).

Let $T^*_{x} M$ be the dual space of $T_{x} M$, called the cotangent space of $M$ at the point $x$. Then the union $T^*M=\bigcup_{x \in M}T^*_{x} M $ denotes the cotangent bundle of $M$. For every fixed point $x \in M$ let $\left( \frac{\partial}{\partial x^i} \right)_{i=\overline{1,n}}$ be the canonical basis of the tangent space $T_xM$, and $\left( d x^i \right)_{i=\overline{1,n}}$ be the dual basis of $T^*_xM$, where $(x^i)_{i=\overline{1,n}}$ is a local coordinate system.

In the sequence by $(M,F)$ we mean a Finsler metric measure manifold $(M,F,\mathsf{m})$, i.e. a Finsler manifold $(M,F)$ endowed with a smooth measure $\mathsf{m}$. If the function $x \mapsto \sigma_F(x)$ denotes the density function of $\mathsf{m}$ in a local coordinate system $(x^i)_{i=\overline{1,n}}$, then we can define the volume form 
\begin{equation} \label{Hausdorff_measure}
\mathrm{d}\mathsf{m}(x) = \sigma_F(x) \mathrm{d} x^1 \land \cdots \land \mathrm{d} x^n ,
\end{equation}
which is used throughout the paper.  
The Finslerian volume of a subset $\Omega \subset M$ is defined as $\mathrm{Vol}_F(\Omega) = \int_{\Omega} \mathrm{d}\mathsf{m}$.

The mean distortion of $(M,F)$ is defined by  
$$ \mu: TM \setminus \{0\} \to (0, \infty) , \quad \mu(x,y) = \frac{\sqrt{\text{det}\big(g_{ij}(x,y)\big)}}{\sigma_F(x)} ,$$
while the mean covariation is defined as
$$\textbf{S}: TM \setminus \{0\} \to \mathbb{R}, \quad 
\textbf{S}(x,y) = \frac{d}{dt} \Big( \text{log} ~\mu \big(\gamma(t), \dot \gamma(t) \big) \Big) \Big\rvert_{t=0} ,$$
where $\gamma$ is the geodesic with $\gamma(0) = x$ and $\dot\gamma(0) = y$.
If $\textbf{S}(x,y) = 0$ on all $TM \setminus \{0\}$, then we say that $(M,F)$ has vanishing mean covariation and we denote it by $\textbf{S} = 0$. 

Finally, let us recall the notion of weighted Ricci curvature introduced by Ohta \cite{Ohta}. For further details we also refer to Ohta and Sturm \cite{Ohta-Sturm} and Xia \cite{Xia2}.

Let $x \in M$ be a point, $y \in T_xM$ a unit vector, $\varepsilon > 0$ and   
$\gamma: [-\varepsilon, \varepsilon] \rightarrow M$ be a geodesic such that $\gamma(0) = x$ and $\dot \gamma(0) = y$. 
Then one can write that $\mathsf m = e^{-\Psi} \mathrm d \mathrm{Vol}_{\dot \gamma}$ along $\gamma$, 
where $\mathrm{Vol}_{\dot \gamma}$ is the volume form of the Riemannian metric $g_{\dot \gamma}$ . 
Then the weighted Ricci curvature is defined as
\begin{align*}
\mathrm{Ric}_n(y) &= \left\{
\begin{array}{lll}
\mathrm{Ric}(y) + (\Psi \circ \gamma)''(0) \ & \text{if }~ (\Psi \circ \gamma)'(0) = 0 ,\\
- \infty  & \text{otherwise};
\end{array}\right. \\
\mathrm{Ric}_N(y) &= \mathrm{Ric}(y) + (\Psi \circ \gamma)''(0) - \frac{(\Psi \circ \gamma)'(0)^2}{N-n}, ~\text{ where } N \in (n, \infty) ;\\
\mathrm{Ric}_{\infty}(y) &= \mathrm{Ric}(y) + (\Psi \circ \gamma)''(0) . 
\end{align*}  
Also, for every $\lambda \geq 0$ and $N \in [n, \infty]$ we define $\mathrm{Ric}_N(\lambda y) = \lambda^2 \mathrm{Ric}_N(y) $.

\subsection{Polar and Legendre transforms} \label{polar}

The dual Finsler metric $F^*$ on $M$ is defined by
\begin{equation*}  
F^*: T^*M \to [0, \infty) ~,\quad 
F^*(x,\alpha) = \sup \{ \alpha(y) :~ y\in T_xM, F(x,y) = 1 \},
\end{equation*}
and it is called the polar transform of $F$.

Since $F^{*2}(x,\cdot)$ is twice differentiable on $T_x^*M\setminus \{0\}$, one can define the Hessian matrix 
$\Big( g^*_{ij}(x,\alpha) \Big) = \left( \left[ \frac{1}{2}F^{*2}(x, \alpha)\right] _{\alpha^{i}\alpha^{j}} \right)$
for every $\alpha = \sum_{i=1}^n \alpha^i dx^i \in T_x^*M\setminus \{0\}$.

Using the strict convexity of the functions $F(x,\cdot), x \in M$, the Legendre transform $J^*:T^*M \to TM$ is defined in the following way. For each $x \in M$, one can assign to every $\alpha \in T_x^*M$ the unique maximizer $ y \in T_xM $ of the mapping 
$$ y ~ ~ \mapsto ~ ~ \alpha(y) - \frac{1}{2}F^2(x,y) .$$ 
Note that if $ J^*(x, \alpha) = (x,y) $, then
\begin{equation}  \label{Legendre_transform}
F(x,y) = F^*(x,\alpha) \quad \text{ and } \quad \alpha(y) = F^*(x,\alpha) F(x,y).
\end{equation}
One can also define the function $J: TM \to T^*M$ in a similar manner and it can be proven that $J^* = J^{-1}$.
Further properties of the Legendre transform can be found, for instance, in Bao, Chern and Shen \cite[Section~14.8]{BCS}, and Ohta and Sturm \cite{Ohta-Sturm}, we just mention the following result.

For every $\alpha = \sum_{i=1}^n \alpha^i dx^i \in T_x^*M$ and every $y = \sum_{i=1}^n y^i \frac{\partial}{\partial x^i} \in T_xM$ we have that
\begin{equation*}  
J^*(x,\alpha) = \sum_{i=1}^n \frac{\partial}{\partial \alpha_i}\left(\frac{1}{2} F^{*2}(x,\alpha)\right)\frac{\partial}{\partial x^i}  
 \text{ and } 
J(x, y) = \sum_{i=1}^n \frac{\partial}{\partial y_i}\left(\frac{1}{2} F^{2}(x,y)\right) dx^i . 
\end{equation*}

\subsection{Gradient vector, Finsler-Laplacian, Sobolev spaces}

Let $u: M \to \mathbb{R}$ be a weakly differentiable function. 
Then for every regular point $x \in M$, $Du(x) \in T_x^*M$ denotes the differential of $u$ at $x$, while the gradient of $u$ at $x$ is defined as
\begin{equation*}  
\boldsymbol{\nabla} u(x) = J^*(x, Du(x)).
\end{equation*}
Note that by relation \eqref{Legendre_transform}, we have
\begin{equation}  \label{metric-co-metric2}
F^*(x, Du(x)) = F(x, \boldsymbol{\nabla} u(x)).
\end{equation}
Using local coordinates, it follows that 
\begin{equation}  \label{gradient_local}
Du(x) = \sum_{i=1}^n \frac{\partial u}{\partial x^i}(x) dx^i 
\quad \text{and} \quad 
\boldsymbol{\nabla} u(x)=\sum_{i,j=1}^n g_{ij}^*(x,Du(x))\frac{\partial u}{\partial x^i}(x)\frac{\partial}{\partial x^j}.
\end{equation}
Therefore, the nonlinearity of the Legendre transform induces the nonlinearity of the gradient operator $\boldsymbol{\nabla}$.

In order to define the Sobolev spaces on the Finsler manifold $(M,F)$, we use the volume form $\mathrm{d}\mathsf{m}$ defined in \eqref{Hausdorff_measure}.
Let $\Omega$ be an open subset of $M$.

The spaces $L^p_{\mathrm{\mathrm{loc}}}(\Omega)$ and $W_{\mathrm{\mathrm{loc}}}^{1,p}(\Omega), p \in [1, \infty]$ are defined in a natural manner, independent of the Finsler structure $F$ and the measure $\mathsf{m}$.
The Sobolev spaces on $(M,F)$, however, are determined by the choices of $F$ and $\mathsf{m}$, i.e.
\begin{equation*}
W_F^{1,2}(\Omega) = 
\left\{ u\in W_{\mathrm{\mathrm{loc}}}^{1,2}(\Omega) ~:~
\int_{\Omega}F^{* 2}(x,Du(x)) ~\mathrm{d}\mathsf{m} < \infty \right\} ,
\end{equation*}
and $W_{0,F}^{1,2}(\Omega)$ is the closure of $C_{0}^{\infty
}(\Omega)$ with respect to the norm
\begin{equation*}
\Vert u\Vert _{W^{1,2}_F(\Omega)} = \left( \int_{\Omega} u^2(x)~\mathrm{d}\mathsf{m}\right)^{\frac{1}{2}} + \left( \int_{\Omega}F^{\ast 2}(x,Du(x))~\mathrm{d}\mathsf{m}\right)^{\frac{1}{2}} ,
\end{equation*}
see Ohta and Sturm \cite{Ohta-Sturm}. In the following we may omit the parameter $x$ for the simplicity of the notation.

Let $X$ be a weakly differentiable vector field on $\Omega$. The divergence of $X$ is defined in a distributional sense, i.e. $\mathrm{div} X: \Omega \to \mathbb{R}$ such that 
\begin{equation}  \label{Greendiv}
\int_{\Omega} \varphi \mathrm{div} X ~{\text d}{\mathsf m} = 
-\int_{\Omega} D\varphi (X) ~\mathrm{d}\mathsf{m} ,
\end{equation}
for every $\varphi \in C^{\infty}_0(\Omega)$, see Ohta and Sturm \cite{Ohta-Sturm}. 
We say that $X \in L_{\mathrm{loc}}^1(\Omega)$ if the function $F(X) \in L_{\mathrm{loc}}^1(\Omega)$. 

The Finsler-Laplace operator $\boldsymbol{\Delta}$ is defined in a distributional sense as
$\boldsymbol{\Delta} u = \mathrm{div}(\boldsymbol{\nabla} u)$
for every function $u \in W^{1,2}_{\mathrm{\mathrm{loc}}}(\Omega)$, i.e. 
\begin{equation*} 
\int_{\Omega} \varphi \boldsymbol{\Delta} u ~\mathrm{d}\mathsf{m} = 
-\int_{\Omega} D\varphi (\boldsymbol{\nabla} u) ~\mathrm{d}\mathsf{m},
\end{equation*}
for all $\varphi \in C^{\infty}_0(\Omega)$. 
Note that in general, the Finsler-Laplace operator $\boldsymbol{\Delta}$ is a nonlinear operator.

Now let $X \in L_{\mathrm{loc}}^1(\Omega)$ be a vector field and $f \in L_{\mathrm{loc}}^1(\Omega)$ a function. We say that 
$f \leq -\mathrm{div} X$ if the inequality holds in the distributional sense, i.e. 
\begin{equation*} 
\int_{\Omega} \varphi f ~\mathrm{d}\mathsf{m} \leq 
-\int_{\Omega} \varphi \mathrm{div} X ~\mathrm{d}\mathsf{m} = 
\int_{\Omega} D\varphi (X) ~\mathrm{d}\mathsf{m},
\end{equation*}
for every nonnegative $\varphi \in C_0^\infty(\Omega)$.

Finally, we say that a function $u \in  W^{1,2}_{\mathrm{loc}}(\Omega)$ is superharmonic in weak sense if 
\begin{equation} \label{superharmonic} 
-\boldsymbol{\Delta} u \geq 0 ~\text{ on }~ \Omega ,
\end{equation}
meaning that 
\begin{equation}\label{superharmonicdef}
\int_\Omega D \varphi(\boldsymbol{\nabla}u) ~\mathrm{d}\mathsf{m} \geq 0 ,
\end{equation} 
for every nonnegative function $\varphi \in C_0^\infty(\Omega)$.

\subsection{Distance function}

Let $\gamma: [a,b] \to M$ be a piecewise differentiable curve. We define the length of the segment $\gamma \big\rvert_{[a,b]}$ as
$$ L_F(\gamma) = \int_a^b F( \gamma(t), \dot \gamma(t)) ~ dt . $$

The distance function $d_F: M \times M \to [0, \infty)$ is defined on $(M,F)$ by
$$ d_F(x_1, x_2)  = \inf_\gamma  L_F(\gamma) ,$$
where the infimum is taken over the set of all piecewise differentiable curves $\gamma: [a,b] \to M$  
such that $\gamma(a) = x_1$ and $\gamma(b) = x_2$.
It can be proven that $d_F(x_1, x_2) = 0$ if and only if $x_1 = x_2$ and that $d_F$ verifies the triangle inequality.
However, in general, the distance function is not symmetric. In fact, we have that $d_F(x_1, x_2) = d_F(x_2, x_1), $ for all $x_1, x_2 \in M$ if and only if $(M,F)$ is a reversible Finsler manifold.

The open forward and backward geodesic balls of center $x_0 \in M$ and radius $R > 0$ are defined by
$$B^+_R(x_0) = \{ x \in M : d_F(x_0, x) < R \} \text{ and } B^-_R(x_0) = \{ x \in M : d_F(x, x_0) < R \}, $$ respectively.

For any fixed point $x_0 \in M$, one can introduce the distance function 
$r: M \to [0, \infty)$, $r(x) = d_F(x_0, x) $, for all $x \in M$.
Then, by Shen \cite[Lemma~3.2.3]{Shen1}, we have that 
\begin{equation}\label{eikonal}
F(x, \boldsymbol{\nabla}r(x)) = F^*(x, Dr(x)) = Dr(x)(\boldsymbol{\nabla}r(x)) = 1 ,  
\end{equation}   
for every $x \in M \setminus (\{x_0\} \cup \text{Cut}(x_0))$,
where $\text{Cut}(x_0)$ denotes the cut locus of the point $x_0$, see Bao, Chern and Shen \cite[\mbox{Chapter} 8]{BCS}.

Furthermore, we have the following comparison theorem for the Finsler-Laplacian of the distance function $r$, see Wu and Xin \cite{WuXin}.

\begin{thm}{\emph{(Laplacian comparison theorem \cite[Theorem 5.1]{WuXin})}} \label{comparison-laplace}
	Let $(M,F)$ be an $n$-dimensional Finsler-manifold with $\mathbf{S}=0$, and suppose that the flag curvature of $M$ is bounded from above, i.e. $\textbf{K} \leq c$, $c \in \mathbb{R}$. Let $r = d_F(x_0, \cdot)$ be the distance function from a fixed point $x_0 \in M$. Then
	$$\boldsymbol{\Delta}r \geq (n-1) {\bf ct}_c(r) $$ 
	for every point $x \in M \setminus (\{x_0\} \cup \emph{Cut}(x_0))$,
	where ${\bf ct}_{c}: (0, \infty) \to \mathbb{R}$, 
	$${\bf ct}_{c}(t) = \left\{
	\begin{array}{lll}
	\sqrt{c} \cot(\sqrt{c}t)  & \hbox{if} & {c} > 0, \\
	\frac{1}{t}                & \hbox{if} & {c} = 0, \\
	\sqrt{-c}\coth(\sqrt{-c}t) & \hbox{if} & {c} < 0.
	\end{array}\right.$$
\end{thm}

\section{Hardy inequalities for superharmonic weight functions on Finsler manifolds} \label{hardy}

In the following let $(M, F)$ be a forward complete $n$-dimensional Finsler manifold and let $\Omega \subset M$ be an open set. 

Inspired by D'Ambrosio and Dipierro \cite{DAmbrosio}, in this section we consider a nonnegative weight function $\rho \in W_{\mathrm{loc}}^{1,2}(\Omega)$, 
and we prove that the superharmo\-nicity of $\rho$ provides a sufficient condition to obtain weighted Hardy inequalities on the Finsler manifold $(M, F)$.

In order to be more self-contained, we state the following lemma, which is crucial for the proof of Theorem \ref{Hardy1} and \ref{Hardy2}. 
The technique of the proof is analogous to the result of Zhao \cite[Theorem 3.1]{Zhao}, 
and is based on relations \eqref{metric-co-metric2}, \eqref{Greendiv} and the H\"{o}lder inequality.

\begin{lem}\label{basiclemma}
	Let $X \in L_{\mathrm{loc}}^1(\Omega)$ be a vector field and $f_X \in L_{\mathrm{loc}}^1(\Omega)$ a nonnegative function such that the following properties hold: 
	\begin{enumerate}[label=(\roman*)]
		\item $f_X \leq -\mathrm{div}X$;
		\item $\frac{F^2(X)}{f_X} \in L_{\mathrm{loc}}^1(\Omega)$.
	\end{enumerate}
	Then every function $u\in C_0^\infty(\Omega)$ satisfies
	\begin{equation}\label{basiclemmaineq1}
	\int_{\Omega} u^2 f_X ~\mathrm{d}\mathsf{m} \leq 
	4 \int_{\Omega} \frac{F^2(x, X)}{f_X}F^2(x, \boldsymbol{\nabla}u) ~\mathrm{d}\mathsf{m} .
	\end{equation}
\end{lem}

By carefully choosing the vector field $X$ and the function $f_X$, 
we can deduce the \mbox{following} weighted Hardy inequality.  
Note that by introducing the reversibility constant $r_F$, we obtain the quantitative analogue of Zhao \cite[Theorem 4.1]{Zhao}.

\begin{thm}\label{Hardy2}
	Let $\rho \in W_{\mathrm{loc}}^{1,2}(\Omega)$ be a nonnegative function and $\theta \in \mathbb{R}$ a constant with the following properties:
	\begin{enumerate}[label=(\roman*)]
		\item \label{superalf} $-(1-\theta) \boldsymbol{\Delta}\rho \geq 0$ on $\Omega$ in weak sense;
		\item \label{loc1} $\frac{F^{*2}(D\rho)}{\rho^{2-\theta}}, ~\rho^{\theta}\in L^1_{\mathrm{loc}}(\Omega).$
	\end{enumerate}
	If $\theta \leq 1$, then 
	\begin{equation*}
	\frac{(1-\theta)^2}{4} \int_{\Omega} \rho^{\theta}\frac{u^2}{\rho^2}F^{* 2}(x, D\rho) ~\mathrm{d}\mathsf{m} 
	\leq \int_{\Omega}\rho^{\theta}F^{2}(x, \boldsymbol{\nabla}u) ~\mathrm{d}\mathsf{m} , 
	\quad \forall u \in C_0^{\infty}(\Omega),
	\end{equation*}
	whereas if $\theta > 1$ and $r_F < \infty$, $r_F$ being the reversibility constant of $(M,F)$, then 
	\begin{equation*}
	\frac{(1-\theta)^2}{4 r_F^2} \int_{\Omega} \rho^{\theta}\frac{u^2}{\rho^2}F^{* 2}(x, D\rho) ~\mathrm{d}\mathsf{m} 
	\leq \int_{\Omega}\rho^{\theta}F^{2}(x, \boldsymbol{\nabla}u) ~\mathrm{d}\mathsf{m} , 
	\quad  \forall u \in C_0^{\infty}(\Omega).
	\end{equation*}
\end{thm}

\begin{proof}	
	The proof is based on the application of Lemma \ref{basiclemma}. Notice that the case $\theta = 1$ is trivial.
	
	Let $\alpha \in (0,1)$, $\rho_{\alpha} = \rho + \alpha > 0$ on $\Omega$, and define the vector field $X$ and the function $f_X$ on $\Omega$ as	
	\begin{equation} \label{X_and_f_X}
	X = (1-\theta) \frac{\boldsymbol{\nabla}\rho_{\alpha}}{\rho_{\alpha}^{1-\theta}} \quad \text{and} \quad 
	f_X = (1-\theta)^2\frac{F^{* 2}(D\rho_{\alpha})}{\rho_{\alpha}^{2-\theta}}.
	\end{equation}	
	Since $\rho^{\theta} \in L^1_{\mathrm{loc}}(\Omega)$, $\frac{1}{\rho_\alpha} \leq \frac{1}{\alpha}$ and $D\rho_\alpha = D\rho$, we have that 
	$X$ and $f_X$ $\in L_{\mathrm{loc}}^1(\Omega)$. 
	Also, by direct calculations we obtain 
	$$\frac{F^2(x, X)}{f_X} = \rho_{\alpha}^\theta 
	\frac{ F^2(x, (1-\theta) \boldsymbol{\nabla} \rho_{\alpha})}{ (1-\theta)^2 F^{*2}(x, D\rho_{\alpha})} .$$
	If $\theta < 1$, then we can write 
	\begin{equation} \label{fraction1}
	\frac{F^2(x, X)}{f_X} = \rho_{\alpha}^\theta 
	\frac{ (1-\theta)^2 F^2(x, \boldsymbol{\nabla} \rho_{\alpha})}{ (1-\theta)^2 F^{*2}(x, D\rho_{\alpha})} 
	= \rho_{\alpha}^\theta \in L_{\mathrm{loc}}^1(\Omega).
	\end{equation}
	However, when $\theta > 1$ and $r_F < \infty$, using the definition \eqref{reversibility_constant} of the reversibility constant, we have 
	\begin{equation} \label{fraction2}
	\frac{F^2(x, X)}{f_X} = \rho_{\alpha}^\theta 
	\frac{  F^2(x, (1-\theta) \boldsymbol{\nabla} \rho_{\alpha})}{ (\theta-1)^2 F^{*2}(x, D\rho_{\alpha})} 
	=\rho_{\alpha}^\theta 
	\frac{  F^2(x, -(\theta-1) \boldsymbol{\nabla} \rho_{\alpha})}{ F^{*2}(x, (\theta-1) D\rho_{\alpha})} 
	\leq r_F^2 \rho_{\alpha}^\theta ,
	\end{equation}
	thus $\frac{F^2(X)}{f_X} \in L_{\mathrm{loc}}^1(\Omega)$.	
	It remains to prove that $f_X \leq -\mathrm{div} X$, for which we refer to Zhao \cite[Theorem 4.1]{Zhao}.
\end{proof}

On the one hand, applying Theorem \ref{Hardy2} with the particular case $\theta = 0$ yields Theorem \ref{Hardy1}. 
On the other hand, by choosing $\theta = 2+q, q > -1$, we obtain the following $L^2$-Caccioppoli-type \mbox{inequality}.

\begin{cor}
	Let $(M, F)$ be a complete Finsler manifold with $r_F < \infty$, and let $\Omega \subset M$ be an open set. 
	Let $\rho \in W_{\mathrm{loc}}^{1,2}(\Omega)$ be a nonnegative function such that $\boldsymbol{\Delta}\rho \geq 0$ on $\Omega$ in weak sense. 
	If $q > -1$ such that $\rho^q F^{*2}(D\rho)$ and $\rho^{2+q} \in L^1_{\mathrm{loc}}(\Omega)$, then we have
	\begin{equation} \label{caccioppoli}
	\frac{(1+q)^2}{4 r_F^2} \int_{\Omega} \rho^{q} F^{*2}(x,D\rho) ~ u^2 ~\mathrm{d}\mathsf{m} 
	\leq \int_{\Omega}\rho^{2+q} F^{2}(x, \boldsymbol{\nabla} u) ~\mathrm{d}\mathsf{m} , 
	\quad  \forall u \in C_0^{\infty}(\Omega).
	\end{equation}
\end{cor}

Finally, by defining the weight function $\rho$ in Theorem \ref{Hardy2} with the help of the Finslerian distance function $d_F$, 
one can obtain Hardy inequalities on Finsler-Hadamard manifolds having finite reversibility constant.

For this let us consider a Finsler-Hadamard manifold $(M,F)$, 
i.e. $M$ is a simply connected, complete Finsler manifold with non-positive flag curvature $\textbf{K} \leq 0$.
Let $r_F$ and $\mathbf{S}$ denote the reversibility constant and the mean covariation of $(M,F)$, respectively. 

Let $x_0 \in M$ be an arbitrarily fixed point and $r = d_F(x_0, \cdot)$ be the distance function from $x_0$ on $M$.
Note that as $(M,F)$ is a Finsler-Hadamard manifold, we have $\text{Cut}(x_0) = \emptyset$. 

By applying Theorem \ref{Hardy2}, we obtain the following Hardy inequality featuring the reversibility constant $r_F$, 
which can be considered the quantitative version of the result given by Zhao \cite[Theorem 1.2]{Zhao}. 
We sketch the proof for completeness.

\begin{thm} \label{Hadamard1}
	Let $(M,F)$ be an $n$-dimensional Finsler-Hadamard manifold with $n \geq 3$, $r_F < \infty$ and $\mathbf{S}=0$. 
	Let $\alpha \in (-\infty, 1)$, then for every $ u \in C^{\infty}_0(M)$ we have 
	\begin{equation} \label{CartanHad1}
	\frac{(n-2)^2 (1-\alpha)^2}{4 \hspace{0.05cm} r_F^2}  \int_{M} r^{\alpha(2-n)} \frac{u^2}{r^2} ~\mathrm{d} {\mathsf m} 
	~ \leq ~ \int_{M} r^{\alpha(2-n)} F^{2}(x, \boldsymbol{\nabla}u) ~\mathrm{d}\mathsf{m} .
	\end{equation}	
\end{thm}

\begin{proof}
	Let $\Omega = M \setminus \{x_0\}$ be an open set, and define $\rho = r^{2 - n}: \Omega \to [0, \infty)$, where $n = \text{dim} M \geq 3$. 
	We are going to apply Theorem \ref{Hardy2} with the weight function $\rho$.
	
	Clearly, we have $\rho(x) > 0$ for every $x \in \Omega$. 
	Furthermore, by using relations \eqref{reversibility_constant} and \eqref{eikonal} we obtain that 
	\begin{equation} \label{L_1_loc}
	F^{*2}(D\rho) = (n-2)^2 r^{2-2n} F^{*2}(-Dr) \leq 
	(n-2)^2 r_F^2~ r^{2-2n}  ~ \in ~ L^1_{\mathrm{loc}}(\Omega) ,
	\end{equation}
	thus $\rho \in W^{1,2}_{\mathrm{loc}}(\Omega)$. 
	
	Applying relation \eqref{eikonal} yields
	\begin{align*}
	\boldsymbol{\Delta }\rho & = (2-n) ~\mathrm{div} (r^{1-n}\boldsymbol{\nabla} r)\\
	& = (2-n) ~ \big( (1-n) r^{-n} Dr(\boldsymbol{\nabla} r) + r^{1-n} \boldsymbol{\Delta }r \big)  \\
	& = (2-n) ~r^{-n} ( 1-n + r \boldsymbol{\Delta}r ) .
	\end{align*}
	
	From Theorem \ref{comparison-laplace} it follows that $\boldsymbol{\Delta}r \geq \frac{n-1}{r}$ on $\Omega$, thus
	\begin{equation*}
	-(1-\alpha) \boldsymbol{\Delta }\rho = (n-2)(1-\alpha) ~r^{-n}(1-n + r \boldsymbol{\Delta }r) \geq 0 ~ ~ \text{ on }  ~ \Omega.
	\end{equation*}
	
	Similarly to \eqref{L_1_loc}, we can prove that $\frac{F^{*2}(D\rho)}{\rho^{2-\alpha}}$ and 
	$\rho^{\alpha} ~ \in L^1_{\mathrm{loc}}(\Omega)$ , thus we can apply Theorem \ref{Hardy2}, obtaining
	\begin{equation*}
	\frac{(n-2)^2 (1-\alpha)^2}{4}  \int_{\Omega} r^{\alpha(2-n)} \frac{u^2}{r^2} F^{* 2}(x, -D r) ~{\mbox{\text{d}}}{\mathsf m} 
	\leq \int_{\Omega} r^{\alpha(2-n)} F^{2}(x, \boldsymbol{\nabla}u) ~\mathrm{d}\mathsf{m} , 
	\end{equation*}
	for every $u \in C_0^{\infty}(\Omega)$.
	
	Applying the inequality
	\begin{equation}	\label{ineq_r_F} 
	F^{* 2}(x, -Dr) \geq \frac{1}{r_F^2} F^{* 2}(x, Dr) = \frac{1}{r_F^2}  
	\end{equation}
	and noting that the set $\{x_0\}$ has null Lebesgue measure completes the proof.
\end{proof}

By choosing $\alpha = 0$ we recover the Hardy inequality below, which was first obtained by Farkas, Krist\'aly and Varga \cite[Proposition 4.1]{Kristaly1}.

\begin{thm} \label{Hadamard_theorem}
	Let $(M,F)$ be an $n$-dimensional Finsler-Hadamard manifold with $n \geq 3$, $r_F < \infty$ and $\mathbf{S}=0$. 
	Then the Hardy inequality
	\begin{equation} \label{CartanHad2}
	\frac{(n-2)^2 }{4 \hspace{0.05cm} r_F^2}  \int_{M} \frac{u^2}{r^2} ~ \mathrm{d}{\mathsf m} 
	~ \leq ~ \int_{M}  F^{2}(x, \boldsymbol{\nabla}u) ~\mathrm{d}\mathsf{m}  
	\end{equation}	
	holds for every $u \in C_0^{\infty}(M)$.
\end{thm}  

Note that if $(M,F)$ is a reversible Finsler manifold, i.e. $r_F = 1$, then the constant $\frac{(n-2)^2 }{4}$ is sharp and never achieved, 
see Farkas, Krist\'aly and Varga \cite{Kristaly1}. 
However, if we let $r_F \rightarrow \infty$, the inequality \eqref{CartanHad2} becomes trivial.

Finally, we have the following logarithmic Hardy inequality. 

\begin{thm}
	Let $(M,F)$ be an $n$-dimensional Finsler-Hadamard manifold with $n \geq 2$, $r_F < \infty$ and $\mathbf{S}=0$, 
	and consider a fixed number $\alpha \in \mathbb{R} \setminus \{1\}$.  
	If $\alpha < 1$ define $\Omega = r^{-1}(0,1)$, while if $\alpha > 1$ let $\Omega = r^{-1}(1,\infty)$. Then we have
	\begin{equation} \label{CartanHad4}
	\frac{(1-\alpha)^2}{4 \hspace{0.05cm} r_F^2} \int_{\Omega} |\log r|^{\alpha}\frac{u^2}{(r \log r)^2} ~\mathrm{d}\mathsf{m}
	\leq \int_{\Omega} |\log r|^{\alpha} F^{2}(x, \boldsymbol{\nabla} u) ~\mathrm{d}\mathsf{m}, ~ \forall  u\in C^{\infty}_0(\Omega).
	\end{equation}
	
	If we set $\alpha = 0$ and $\Omega = r^{-1}([0,1))$, then we have
	\begin{equation} \label{CartanHad3}
	\frac{1}{4 \hspace{0.05cm} r_F^2} \int_{\Omega} \frac{u^2}{(r \log r)^2} ~\mathrm{d}\mathsf{m}
	\leq \int_{\Omega} F^{2}(x, \boldsymbol{\nabla} u) ~\mathrm{d}\mathsf{m}, \quad \forall  u\in C^{\infty}_0(\Omega).
	\end{equation}
\end{thm}

\begin{proof}
	Let $\rho = (\alpha-1) \log r: \Omega \to \mathbb{R}$. Clearly, in both cases $\alpha < 1$ and $\alpha > 1$ we have that $\rho > 0$ on $\Omega$. 
	Moreover, similarly to the proof of Theorem \ref{Hadamard1}, we can prove that $\rho \in W^{1,2}_{\mathrm{loc}}(\Omega)$ 
	and $\frac{F^{*2}(D\rho)}{\rho^{2-\alpha}}$, $\rho^{\alpha} ~ \in L^1_{\mathrm{loc}}(\Omega)$.
	
	Using relation \eqref{eikonal} and Theorem \ref{comparison-laplace}, we obtain that
	\begin{align*}
	-(1-\alpha)\boldsymbol{\Delta }\rho & = (\alpha -1) \mathrm{div} (\boldsymbol{\nabla }\rho)  \\
	& = (\alpha -1 )^2  ~\mathrm{div} \left(\frac{1}{r} \boldsymbol{\nabla } r\right) \\
	& = (\alpha -1 )^2 \left(-\frac{1}{r^2} + \frac{\boldsymbol{\Delta }r}{r} \right) \\
	&\geq (\alpha -1 )^2 ~ \frac{n-2}{r^2} ~ \geq ~ 0,
	\end{align*}
	so we can apply Theorem \ref{Hardy2}. 
	
	If $\alpha > 1$, by using relation \eqref{eikonal} we obtain \eqref{CartanHad4}.  
	If $\alpha < 1$, applying Theorem \ref{Hardy2} results in 
	\begin{equation*}
	\frac{(1-\alpha)^2}{4} \int_{\Omega} (-\log r)^{\alpha}\frac{u^2}{(r \log r)^2} F^{*2}(x, -Dr) ~\mathrm{d}\mathsf{m}
	\leq \int_{\Omega} (-\log r)^{\alpha} F^{2}(x, \boldsymbol{\nabla} u) ~\mathrm{d}\mathsf{m}, 
	\end{equation*}
	for every $u \in C^{\infty}_0(\Omega)$.
	
	By applying relation \eqref{ineq_r_F} we obtain inequality \eqref{CartanHad4}.
	
	Now set $\alpha = 0$ and $\Omega = r^{-1}([0,1))$. Using the fact that the set $\{x_0\}$ has null Lebesgue measure completes the proof.  
\end{proof}

\vspace{0.5cm}

In order to study the sharpness of the constants involved in the Hardy inequalities presented above, 
we shall introduce the set $D^{1,2}(\Omega)$, defined by the completion of $C_0^{\infty}(\Omega)$ with respect to the norm
\begin{equation} \label{D_norm}
\|u\|_{D^{1,2}(\Omega)} = \left(\int_{\Omega} F^{*2}(x, Du)~\mathrm{d}\mathsf{m}\right)^{\frac{1}{2}}.
\end{equation}

For the sake of simplicity, we consider the Hardy inequality obtained in Theorem \ref{Hardy1}. 
The constants in Theorem \ref{Hardy2} can be treated in a similar manner.

Since $C_0^{\infty}(\Omega)$ is dense in $D^{1,2}(\Omega)$, the best constant in inequality \eqref{ineqmain} is defined by
\begin{equation}\label{bestconstant}
C(\Omega) = \inf_{ \substack{u \in D^{1,2}(\Omega) \\ u \neq 0} } 
\frac{\int_{\Omega}F^{*2}(x, Du) ~\mathrm{d}\mathsf{m}}{\int_{\Omega} \frac{u^2}{\rho^2} F^{*2}(x, D\rho) ~\mathrm{d}\mathsf{m}}.
\end{equation}
Obviously, we have $\frac{1}{4} \leq C(\Omega)$.

Let us consider a nonnegative weight function $\rho \in W_{\mathrm{loc}}^{1,2}(\Omega)$ satisfying the conditions of Theorem \ref{Hardy1} such that $\rho^{\frac{1}{2}} \in D^{1,2}(\Omega)$.
Then the Hardy inequality \eqref{ineqmain} holds for every function $u \in D^{1,2}(\Omega)$, 
in particular for $\rho^{\frac{1}{2}}$, which yields
\begin{equation*}
\frac{1}{4} \int_{\Omega} \frac{1}{\rho} F^{*2}(x, D\rho) ~\mathrm{d}\mathsf{m}
\leq \int_{\Omega} F^{*2}\left(x, D(\rho^{\frac{1}{2}})\right) ~\mathrm{d}\mathsf{m} .
\end{equation*} 
On the other hand, we have
\begin{equation*}
\int_{\Omega}F^{*2}\left(x, D(\rho^{\frac{1}{2}})\right) ~\mathrm{d}\mathsf{m}  
= \int_{\Omega}  F^{*2} \Big(x, \frac{1}{2} \rho^{-\frac{1}{2}} D\rho \Big) ~\mathrm{d}\mathsf{m}     
= \frac{1}{4} \int_{\Omega} \frac{1}{\rho}  F^{*2}(x, D\rho) ~\mathrm{d}\mathsf{m} 	 . 
\end{equation*}

Thus $C(\Omega) = \frac{1}{4}$ is sharp and  $\rho^{\frac{1}{2}} \in D^{1,2}(\Omega)$ is a minimizer.

The optimality of the constant $\frac{1}{4}$ when $\rho^{\frac{1}{2}} \notin D^{1,2}(\Omega)$ remains an open question and shall be studied in a forthcoming paper.

\section{Gagliardo-Nirenberg inequality and Heisenberg-Pauli-Weyl uncertainty principle on Finsler manifolds}  \label{generalizations}

In this section we present a generalization of Lemma \ref{basiclemma}, which induces a weighted Gagliardo-Nirenberg inequality and a Heisenberg-Pauli-Weyl uncertainty principle on the Finsler manifold $(M,F)$. In the sequel let $(M, F)$ be a forward complete Finsler manifold and $\Omega \subset M$ an open set.

\begin{lem}\label{basiclemma1}
	Let $X \in L_{\mathrm{loc}}^1(\Omega)$ be a vector field on $\Omega$ and $f_X \in L_{\mathrm{loc}}^1(\Omega)$ a nonnegative function such that $f_X \leq -\mathrm{div}X$ and $\frac{F^2( X)}{f_X} \in L_{\mathrm{loc}}^1(\Omega)$.
	Then  we have
	\begin{align}\label{basicgeneralineq1}
	& \int_{\Omega} |u|^s F^q(x,X) ~ \mathrm{d}\mathsf{m} \leq \nonumber \\
	& \leq 
	4^{\frac{1}{p}}\left( \int_{\Omega} \frac{F^2(x, X)}{f_X}F^2(x, \boldsymbol{\nabla}u) ~ \mathrm{d}\mathsf{m}\right)^{\frac{1}{p}}
	\left( \int_{\Omega} \frac{F^{qp'}(x, X)}{f_X^{p'-1}} |u|^{\frac{ps-2}{p-1}} ~\mathrm{d}\mathsf{m}\right)^{\frac{1}{p'}} 
	\end{align}
	for every function $u \in C_0^{\infty}(\Omega)$ and every real numbers $q \in \mathbb{R}$, $s > 0$ and $p, p' > 1$ such that
	$ \frac{1}{p} + \frac{1}{p'} = 1 $.
	
	\begin{proof}
		By applying the H\"{o}lder inequality and Lemma \ref{basiclemma}, we get
		\begin{align*}
		& \int_{\Omega} |u|^s F^q(x,X) ~\mathrm{d}\mathsf{m}
		= \int_{\Omega} |u|^{\frac{2}{p}}f_X^{\frac{1}{p}} ~F^q(x, X)~f_X^{-\frac{1}{p}}|u|^{s-\frac{2}{p}} ~\mathrm{d}\mathsf{m} \\
		& \leq 
		\left( \int_{\Omega} |u|^{2}f_X ~\mathrm{d}\mathsf{m} \right)^{\frac{1}{p}}
		\left(\int_{\Omega} F^{qp'}(x, X)~f_X^{-\frac{p'}{p}}~|u|^{p'(s-\frac{2}{p})} ~\mathrm{d}\mathsf{m} \right)^{\frac{1}{p'}} \\
		& \leq 
		4^{\frac{1}{p}}\left( \int_{\Omega} \frac{F^2(x, X)}{f_X}F^2(x, \boldsymbol{\nabla}u) ~ \mathrm{d}\mathsf{m}\right)^{\frac{1}{p}}
		\left( \int_{\Omega} \frac{F^{qp'}(x, X)}{f_X^{p'-1}} |u|^{\frac{ps-2}{p-1}} ~\mathrm{d}\mathsf{m}\right)^{\frac{1}{p'}} ,
		\end{align*}
		where $\frac{1}{p} + \frac{1}{p'} = 1$.	
	\end{proof}
\end{lem}

We introduce the notation $w = \frac{F( X)}{\sqrt{f_X}}$ .

Choosing $p = 1 + \frac{2}{tz}, ~ t, z > 0 $ in \eqref{basicgeneralineq1} yields 
\begin{equation}\label{basicgeneralineq2}
\left(\int_{\Omega} |u|^s F^q(x,X)~ \mathrm{d}\mathsf{m}\right)^{\frac{1}{s}} 
\leq 2^{\frac{q}{s}}\left( \int_{\Omega} w^2F^2(x, \boldsymbol{\nabla}u) ~\mathrm{d}\mathsf{m}\right)^{\frac{r}{2}}
\left( \int_{\Omega} w^{t z}|u|^{z} ~\mathrm{d}\mathsf{m}\right)^{\frac{1-r}{z}}
\end{equation}
for all $u \in C_0^{\infty}(\Omega)$, where 
$$ \frac{1}{s} = \frac{r}{2}+\frac{1-r}{z},
\quad \frac{1}{q} = \frac{1}{2}+\frac{1}{t z}, \quad r=\frac{t}{1+t}~ \in (0,1) ,$$
while setting $q=0$ in \eqref{basicgeneralineq1} implies
\begin{equation}\label{basicgeneralineq3}
\int_{\Omega} |u|^s  \mathrm{d}\mathsf{m}
\leq 4^{\frac{1}{p}}\left( \int_{\Omega} w^2F^2(x, \boldsymbol{\nabla}u)~ \mathrm{d}\mathsf{m}\right)^{\frac{1}{p}}
\left( \int_{\Omega}\frac{1}{f_X^{p'-1}} |u|^{\frac{ps-2}{p-1}} \mathrm{d}\mathsf{m}\right)^{\frac{1}{p'}},
\end{equation} 
for every $u \in C_0^{\infty}(\Omega)$, where $s>0$ and $p, p'>1$ such that $\frac{1}{p} + \frac{1}{p'} = 1$.

\vspace{0.5cm}

As before, the careful choice of $X$ and $f_X$ in relations \eqref{basicgeneralineq2} and \eqref{basicgeneralineq3} implies 
a weighted Gagliardo-Nirenberg inequality and an uncertainty principle.

\vspace{0.5cm}

In particular, defining $X$ and $f_X$ as in the proof of Theorem \ref{Hardy2} (see relation \eqref{X_and_f_X} and set $\theta = 0$), we obtain that 
$w^2 = 1$, thus inequalities \eqref{basicgeneralineq2} and \eqref{basicgeneralineq3} yield the following theorems.

\begin{thm}
	Let $\rho \in W_{\mathrm{loc}}^{1,2}(\Omega)$ be a nonnegative function such that $\rho$ is superharmonic on $\Omega$ in weak sense.
	Let $q \in \mathbb{R}$, $s, z > 0$ and $r \in (0,1)$. 
	Then 
	\begin{equation*} 
	\left(\int_{\Omega} |u|^s \frac{F^{*q}(x,D\rho)}{\rho^{q}} ~\mathrm{d}\mathsf{m}\right)^{\frac{1}{s}}
	\leq 2^{\frac{q}{s}} \left( \int_{\Omega} F^2(x, \boldsymbol{\nabla}u) ~\mathrm{d}\mathsf{m}\right)^{\frac{r}{2}}
	\left( \int_{\Omega} |u|^{z} ~\mathrm{d}\mathsf{m}\right)^{\frac{1-r}{z}}
	\end{equation*}	
	for every $u \in C^\infty_0(\Omega)$, where 
	\[ \frac{1}{s}=\frac{r}{2}+\frac{1-r}{z} \quad \text{ and } \quad \frac{1}{q}=\frac{1}{2}+\frac{1-r}{rz} .\]	
\end{thm}

Taking $q = 1$ and $s = 2$ yields $r = \frac{1}{2}$ and $z = 2$, thus we obtain Corollary \ref{GN}.

\begin{thm}
	Let $\rho \in W_{\mathrm{loc}}^{1,2}(\Omega)$ be a nonnegative function such that $\rho$ is superharmonic on $\Omega$ in weak sense.
	Let $s > 0$ and $p, p' > 1$ such that $\frac{1}{p} + \frac{1}{p'} = 1$. 
	Then for every $u \in C^\infty_0(\Omega)$ we have
	\begin{equation*} 
	\int_{\Omega} |u|^s ~ \mathrm{d}\mathsf{m}
	\leq 4^{\frac{1}{p}} \left( \int_{\Omega} F^2(x, \boldsymbol{\nabla}u) ~\mathrm{d}\mathsf{m}\right)^{\frac{1}{p}}
	\left( \int_{\Omega} \frac{\rho^{2(p'-1)}}{F^{*2(p'-1)}(x, D\rho)} |u|^{\frac{ps-2}{p-1}} ~\mathrm{d}\mathsf{m}\right)^{\frac{1}{p'}}.
	\end{equation*}	
\end{thm}

In particular, setting $p = s = 2$ implies Corollary \ref{HPW}.
Note that in the Euclidean setting, if we choose $\rho(x) = |x|$ to be the Euclidean norm, 
then we have $|\nabla\rho(x)| = 1$ for every $ x \in \mathbb{R}^n, x \neq 0$, and Corollary \ref{HPW} coincides with the uncertainty principle in the Euclidean space $\mathbb{R}^n$.

\vspace{0.5cm}

\section{Extension of Theorem \ref{Hardy1} on \texorpdfstring{$2$}{2}-hyperbolic Finsler manifolds} \label{extension}

In  this section we extend the validity of the Hardy  inequality \eqref{ineqmain} to functions 
\mbox{$u \in C_0^{\infty}(M)$} by considering $2$-hyperbolic Finsler manifolds having bounded geometry.

Let $(M, F)$ be a forward complete $n$-dimensional Finsler manifold with uniformly convex Finsler metric $F$, i.e. there exist two positive constants 
$\lambda \leq \Lambda < \infty$ such that 
\begin{equation} \label{uniformly_convex}
\lambda F^2(x,y) \leq g_v(y,y) \leq \Lambda F^2(x,y) 
\end{equation}
for all $x \in M$ and $y,v \in T_x M$, $v \neq 0$, where $g$ is the fundamental tensor associated to $F$. 

Let $\Omega \subset M$ be an open set. 
In order to extend Theorem \ref{Hardy1} to the class of functions $C_0^{\infty}(M)$, it is enough to prove the inclusion
\begin{equation}\label{inclusion1}
D^{1,2}(M) \subset  D^{1,2}(\Omega) ,
\end{equation}
where $D^{1,2}(\Omega)$ is defined as the completion of $C_0^{\infty}(\Omega)$ with respect to the norm 
$\| \cdot \|_{D^{1,2}(\Omega)}$, see \eqref{D_norm}.

Indeed, by Theorem \ref{Hardy1}, inequality \eqref{ineqmain} holds for every function $u \in C_0^{\infty}(\Omega)$, 
so it remains true for every $u \in  D^{1,2}(\Omega)$. 
Using the inclusions $C_0^{\infty}(M) \subset D^{1,2}(M) \subset  D^{1,2}(\Omega)$,  
we obtain that the Hardy inequality \eqref{ineqmain} holds for every function  $u \in C_0^{\infty}(M)$.

In order to discuss sufficient conditions under which relation \eqref{inclusion1} is valid, we shall recall the definition of capacity, see Troyanov \cite{troyanov2}. 

Let $K \subset M$ be a compact set. 
The $2$-capacity (or simply, capacity) of $K$ in $M$ is defined by
\begin{align*}
\capac_2(K, M)=\inf \Big \{ \int_{M}F^{*2}(x, Du)\mathrm{d}\mathsf{m} : u \in C^\infty_0(M), u \geq 1 ~\text{on}~ K \Big\},
\end{align*}
Clearly, a truncation argument shows that this is also equivalent with the definition
\begin{align*}
\capac_2(K, M) = \inf \Big \{ \int_{M} F^{*2}(x, Du)\mathrm{d}\mathsf{m} : ~ 
& u \in C_0^\infty(M), 0 \leq u \leq 1 ~\text{on}~ M, \\
& u = 1 \text{ on a neighborhood of } K \Big \}.
\end{align*}

The Finsler manifold $(M,F)$ is called $2$-parabolic if there exists a compact set 
$K \subset M$ with non-empty interior such that $\capac_2(K,M) = 0$.
Note that this definition is equivalent with the fact that $\capac_2(K,M) = 0$ for every compact subset $K \subset M$, 
for the proof see Troyanov \cite[Corollary 3.1]{troyanov2}.

On the other hand, $(M,F)$ is said to be $2$-hyperbolic if there exists a compact set 
$K \subset M$ with non-empty interior such that $\capac_2(K,M) > 0$. 
Again, it can be proven that $(M,F)$ is $2$-hyperbolic if and only if 
the capacity of any compact set $K \subset M$ with non-empty interior is positive. A list of examples of $2$-hyperbolic and $2$-parabolic manifolds can be found in Troyanov \cite[Section 2]{troyanov2}.

Now let $D \subset M$ be a bounded domain. Similarly to Troyanov \cite{troyanov1}, we  
define the Banach space $E(D,M)$ endowed with the norm $\|\cdot\|_{E(D,M)}$ as the set of all functions $u \in W_{\mathrm{loc}}^{1,2}(M)$ such that
\begin{equation*}
\|u\|_{E(D,M)} = \left( \int_{D} u^2\mathrm{d}\mathsf{m}+\int_{M}F^{*2}(x, Du)~\mathrm{d}\mathsf{m} \right)^\frac{1}{2} ~ < ~ \infty~.
\end{equation*}
Also, let $E_0(D,M)$ denote the closure of $C_0^{\infty}(M)$ in $E(D,M)$ with respect to the norm \mbox{$\|\cdot\|_{E(D,M)}$}.

In the following we prove some properties of $2$-hyperbolic Finsler manifolds. 
These can be obtained as natural extensions of the results considering the Riemannian case, presented in Troyanov \cite{troyanov1}.    
First of all, let us recall the well-known Poincar\'e inequality on Riemannian manifolds, see Hebey \cite[Theorem 2.10]{Hebey}.

\begin{thm}{\emph{(\cite[Theorem 2.10]{Hebey})}}\label{poincare_riemann}
	Let $(M, g)$ be a complete Riemannian manifold of dimension $n \geq 3$, and let $K \subset M$ be a compact set. 
	Then there exists a positive constant $C = C(K)$ such that 
	\begin{equation*}
	\left(\int_{K}|u - \overline{u}|^2~ dv_g\right)^{\frac{1}{2}} \leq 
	C \left(\int_{K} | \nabla u|_g^2~dv_g\right)^{\frac{1}{2}} ,
	\end{equation*}
	for all $u \in W_{\mathrm{loc}}^{1,2}(M)$, where 
	$ \overline{u} = \frac{1}{\mathrm{Vol}_g(K)} \int_{K} u ~dv_g $
	denotes the mean value of $u$ on the set $K$.
\end{thm}

Here $dv_g$, $|\cdot|_g$ and $\mathrm{Vol}_g$ stand, respectively, for the Lebesgue volume element of $M$, the norm determined by the Riemannian metric $g$ and the Riemannian volume induced by $g$.

One can extend the Poincar\'e inequality to Finsler manifolds by adding a lower Ricci curvature bound condition. 
In the following let $B_R = B^+_R(x_0)$ denote the forward geodesic ball of center $x_0$ and radius $R > 0$ for some $x_0 \in M$, 
and let
$$ \overline{u} = \frac{1}{\mathrm{Vol}_F(B_R)} \int_{B_R} u ~\mathrm{d}\mathsf{m} $$
denote the mean value of a function $u$ on the set $B_R$.
Then we have the following local uniform Poincar\'e \mbox{inequality}, which was proved by Xia \cite{Xia2}.

\begin{thm}{\emph{(\cite[Theorem 3.2]{Xia2})}}  \label{poincare}
	Let $(M,F)$ be a forward geodesically complete Finsler \mbox{manifold} with uniformly convex Finsler structure $F$, such that the weighted Ricci curvature satisfies 
	$\mathrm{Ric}_N \geq - \kappa$, where $\kappa > 0$.
	Then for every forward geodesic ball $B_R \subset M$ there exists a positive constant $C = C(N, \kappa, \lambda, \Lambda, R)$, depending on $N$, $\kappa$, the uniform constants $\lambda$ and $\Lambda$ in \eqref{uniformly_convex} and the radius $R$, such that 
	\begin{equation*}
	\left(\int_{B_R}|u - \overline{u}|^2~\mathrm{d}\mathsf{m}\right)^{\frac{1}{2}} \leq 
	C \left(\int_{B_R} F^{*2}(x, Du)~\mathrm{d}\mathsf{m}\right)^{\frac{1}{2}} ,
	\end{equation*}
	for all $u \in W_{\mathrm{loc}}^{1,2}(M)$, where 
	$ \overline{u}$ is the mean value of $u$ on $B_R$. 
\end{thm}

Combining this Poincar\'e inequality with the H\"older inequality, we obtain
that, under the conditions of Theorem \ref{poincare}, for every forward geodesic ball $B_R \subset M$ there exists a positive constant $C = C(N, \kappa, \lambda, \Lambda, R)$ such that 
\begin{equation} \label{conspoincare}
\int_{B_R}|u - \overline{u}|~\mathrm{d}\mathsf{m} \leq C \left(\int_{B_R} F^{*2}(x, Du)~\mathrm{d}\mathsf{m}\right)^{\frac{1}{2}},
\end{equation}
for every $u\in W_{\mathrm{loc}}^{1,2}(M)$.

Now we are able to prove some results considering $2$-hyperbolic Finsler manifolds.

\begin{thm}\label{inegimportanta}
	
	Let $(M,F)$ be a forward geodesically complete $2$-hyperbolic Finsler manifold with uniformly convex Finsler structure $F$, such that 
	$\mathrm{Ric}_N \geq - \kappa$, $\kappa > 0$.
	Then for every forward geodesic ball $B_R \subset M$ there exists a positive constant $C = C(N, \kappa, \lambda, \Lambda, R)$ such that 
	\begin{equation*}
	\int_{B_R} |u| ~\mathrm{d}\mathsf{m} ~\leq~ C \left(\int_{M}F^{*2}(x, Du)~\mathrm{d}\mathsf{m}\right)^{\frac{1}{2}}, 
	\end{equation*}	
	for all $u \in E_0(B_R, M) \cap C_0(M)$.
\end{thm}

\begin{proof}
	Suppose by contradiction that such a constant $C$ does not exist. 
	Then for every $\varepsilon > 0$ there exists a function $u \coloneqq u_\varepsilon \in E_0(B_R, M) \cap C_0(M)$ for some $B_R = B^+_R(x_0)$ forward geodesic ball, 
	such that 
	\begin{equation}\label{ineq_eps}
	\int_{B_R} |u|~\mathrm{d}\mathsf{m} = \mathrm{Vol}_F(B_R) ~~\text{ and }~~ 
	\left(\int_{M}F^{*2}(x, Du)~\mathrm{d}\mathsf{m}\right)^{\frac{1}{2}} \leq \varepsilon~.
	\end{equation}
	Since $u$ has compact support, we can assume that $u \geq 0$, otherwise we can replace $u$ by $|u| \in E_0(B_R, M) \cap C_0(M)$. 
	Then we have $\overline{u} = 1$, $\overline{u}$ being the mean value of $u$ on $B_R$. 
	
	Applying inequality \eqref{conspoincare} to $u$, we obtain that there exists a constant $C \geq 0$ such that 
	\begin{equation}\label{majoralas1}
	\int_{B_R}|u-1|~\mathrm{d}\mathsf{m} \leq C \varepsilon.
	\end{equation}
	
	Let $0 < r < R$ and $B_r = B^+_r(x_0) \subset B_R$ be a forward geodesic ball in $B_R$. 
	We choose a function $\varphi \in C_0^{\infty}(M)$ such that $0 \leq \varphi \leq \frac{1}{2}$ 
	with $\supp \varphi \subset B_R$ and $\varphi = \frac{1}{2}$ on $\overline{B_r}$. 	
	Then one can define
	$$v \coloneqq v_\varepsilon = 2 \max \{u, \varphi\}  \in  C_0(M).$$	
	Clearly, we have $v \geq 1$ on $\overline{B_r}$. 	
	Furthermore, let $\overline{B}$ be a closed forward geodesic ball such that $\supp v \cup B_R \subset \overline{B}$. Then, due to the compactness of $\overline{B}$, it follows that $v \in W^{1,2}_{0,F}(\overline{B}) \subset W^{1,2}_{0,F}(M)$.

	Now let us introduce the sets
	\begin{equation*}
	A_1 = \{x \in B_R : ~ \varphi(x) \geq u(x)\}\quad \textrm{ and } \quad A_2 = \left\{x \in B_R : ~ |u(x)-1| \geq \frac{1}{2} \right\}.
	\end{equation*}
	Then, for every $x \in A_1$ we have  $u(x) - 1 \leq -\frac{1}{2}$, which means that $A_1 \subset A_2$. 
	Thus we have
	\begin{equation*}
	\int_{A_2} |u-1| ~\mathrm{d}\mathsf{m}\geq \frac{1}{2} \textrm{Vol}_F(A_2) \geq \frac{1}{2} \textrm{Vol}_F(A_1),
	\end{equation*}
	which implies by relation \eqref{majoralas1} that
	\begin{equation}\label{volumeA}
	\textrm{Vol}_F(A_1) \leq 2 C \varepsilon. 
	\end{equation}
	
	On the other hand, we have almost everywhere that
	$$Dv=\left\{
	\begin{array}{lll}
	2Du \     & \hbox{on} &  M \setminus A_1, \\
	2D\varphi & \hbox{on} & A_1,
	\end{array}\right.$$
	which, together with \eqref{ineq_eps} and \eqref{volumeA}, implies that 
	\begin{align} \label{masik_veges}
	\int_{M}F^{*2}(x, Dv)~\mathrm{d}\mathsf{m} 
	& = 4 \int_{A_1}F^{*2}(x, D\varphi)~\mathrm{d}\mathsf{m} + 4 \int_{M \setminus A_1}F^{*2}(x, Du)~\mathrm{d}\mathsf{m}   \nonumber \\
	& \leq 4 \left(\sup_{x \in A_1} F^*(x,D\varphi)\right)^2 \textrm{Vol}_F(A_1) + 4 \int_{M}F^{*2}(x, Du)~\mathrm{d}\mathsf{m} \nonumber \\
	& \leq 8 C \varepsilon \cdot \left(\sup_{x \in A_1} F^*(x, D\varphi)\right)^2 + 4 \varepsilon^2 .
	\end{align}
	
	As $\varphi \in C_0^{\infty}(M)$, letting $\varepsilon \rightarrow 0$ in \eqref{masik_veges} yields that
	$$	\inf_{\varepsilon > 0} \int_{M}F^{*2}(x, Dv_\varepsilon)\mathrm{d}\mathsf{m}  = 0,$$
	i.e. $\capac_2(\overline{B_r}, M) = 0$. It follows that $(M, F)$ is $2$-parabolic, which is a contradiction.
\end{proof}

Next, we can prove the following inequality.

\begin{thm}\label{troyanov3}
	Let $(M,F)$ be a forward geodesically complete $2$-hyperbolic Finsler ma\-ni\-fold with uniformly convex Finsler structure $F$, 
	satisfying $\mathrm{Ric}_N \geq -\kappa$, $\kappa > 0$, 
	and let $K \subset M$ be a compact set. 
	Then there exists a positive constant $C = C(N, \kappa, \lambda, \Lambda, K)$ such that 
	\begin{equation*}
	\left( \int_{K} u^2 ~\mathrm{d}\mathsf{m} \right)^\frac{1}{2} ~\leq~ 
	C \left(\int_{M}F^{*2}(x, Du)~\mathrm{d}\mathsf{m}\right)^{\frac{1}{2}},
	\end{equation*}
	for all $u \in C_0^\infty(M)$.
\end{thm}

\begin{proof}
	Let $u \in C_0^\infty(M)$ and  $B_R \subset M$ be a forward geodesic ball of radius $R$ such that  $K \subset B_R$. 
	Define $\overline{u}$ to be the mean value of $u$ on $B_R$.
	
	By applying Theorems \ref{poincare} and \ref{inegimportanta}, it follows that there exist the constants $C_1, C_2 > 0$ such that
	\begin{align*}
	\|u\|_{L^2(K)} & \leq \|u\|_{L^2(B_R)} \leq \|u- \overline{u}\|_{L^2(B_R)} + \| \overline{u} \|_{L^2(B_R)}  \\
	&\leq C_1 \left(\int_{B_R} F^{*2}(x, Du) ~\mathrm{d}\mathsf{m}\right)^{\frac{1}{2}}+| \overline{u} |\mathrm{Vol}_F(B_R)^{\frac{1}{2}}  \\
	&\leq C_1 \left(\int_{M} F^{*2}(x, Du) ~\mathrm{d}\mathsf{m}\right)^{\frac{1}{2}} + 
	\mathrm{Vol}_F(B_R)^{-\frac{1}{2}}\int_{B_R} |u| ~\mathrm{d}\mathsf{m} \\
	&\leq \left( C_1 + 
	C_2 \mathrm{Vol}_F(B_R)^{-\frac{1}{2}} \right)   \left(\int_{M} F^{*2}(x, Du) ~\mathrm{d}\mathsf{m}\right)^{\frac{1}{2}} .
	\end{align*}
\end{proof}

By using Theorem \ref{troyanov3} we can prove the following inclusion.

\begin{thm}\label{theorem_inclusion}
	Let $(M,F)$ be a geodesically complete $2$-hyperbolic Finsler manifold with uniformly convex Finsler structure $F$, $r_F < \infty $ and $\mathrm{Ric}_N \geq -\kappa$, $\kappa > 0$. 
	Let $K \subset M$ be a compact set with $\capac_2(K,M) = 0$. Then the inclusion
	$D^{1,2}(M) \subset D^{1,2}(M \setminus K)$ holds.
\end{thm}

\begin{proof}
	Let $\phi \in C_0^{\infty}(M)$. As $D^{1,2}(M)$ is the completion of $C_0^{\infty}(M)$ with respect to the norm $\| \cdot \|_{D^{1,2}(M)}$, 
	it is sufficient to prove that $\phi \in D^{1,2}(M \setminus K)$.
	
	As $\capac_2(K,M) = 0$, it follows that there exists a sequence $(u_k)_{k \in \mathbb{N}}$ in $C_0^{\infty}(M)$ such that $0 \leq u_k \leq 1$, 
	$u_k = 1$ on a neighborhood of $K$ and $\| u_k \|_{D^{1,2}(M)} \rightarrow 0$ as $ k \rightarrow \infty$.
	
	For every $k \in \mathbb{N}$ define $\phi_k = (1-u_k) \phi$. 
	It is clear that $\phi_k \in C_0^{\infty}(M \setminus K)$, $\forall k \in \mathbb{N}$. 
	Now we are going to prove that $\phi_k \rightarrow \phi$ in $\| \cdot \|_{D^{1,2}(M \setminus K)}$. 
	Indeed, we have that
	\begin{align}\label{utolso_egyenlotlenseg}
	&\left(\int_{M\setminus K} F^{*2}(x, D\phi_k - D\phi) ~\mathrm{d}\mathsf{m}\right)^{\frac{1}{2}} \nonumber \\
	&=\left(\int_{M\setminus K} F^{*2}(x, -D\phi u_k - \phi Du_k)~\mathrm{d}\mathsf{m}\right)^{\frac{1}{2}} \nonumber \\
	&\leq r_F \left \{ \left(\int_{M\setminus K} u_k^2F^{*2}(x, D\phi)~\mathrm{d}\mathsf{m}\right)^{\frac{1}{2}} 
	+ \left(\int_{M\setminus K} \phi^2F^{*2}(x,  Du_k)~\mathrm{d}\mathsf{m}\right)^{\frac{1}{2}} \right \}.  
	\end{align}
	
	On the one hand, since $\phi \in C_0^{\infty}(M)$ and $\| u_k \|_{D^{1,2}(M)} \rightarrow 0$ as $ k \rightarrow \infty$, it follows that
	\begin{align*}
	\left(\int_{M\setminus K} \phi^2F^{*2}(x,  Du_k)~\mathrm{d}\mathsf{m}\right)^{\frac{1}{2}}
	\leq ~ \sup_{x \in M} |\phi(x)| \left(\int_{M} F^{*2}(x,  Du_k)~\mathrm{d}\mathsf{m}\right)^{\frac{1}{2}} 
	\longrightarrow ~0
	\end{align*}
	when $ k \rightarrow \infty$.
	
	On the other hand, denoting by $S \coloneqq \supp \phi$ the compact support of $\phi$ and applying Theorem \ref{troyanov3}, we obtain that there exists a constant $C>0$ such that
	\begin{align*}
	& \left(\int_{M\setminus K} u_k^2F^{*2}(x, D\phi)~\mathrm{d}\mathsf{m}\right)^{\frac{1}{2}} 
	\leq \left(\int_{S} u_k^2F^{*2}(x, D\phi)~\mathrm{d}\mathsf{m}\right)^{\frac{1}{2}} \\
	& \leq ~\sup_{x \in S} F^*(x,D\phi) \left(\int_{S} u_k^2~\mathrm{d}\mathsf{m}\right)^{\frac{1}{2}} \\
	& \leq C~\sup_{x \in S} F^*(x,D\phi)~ \left(\int_{M}F^{*2}(x, Du_k)~\mathrm{d}\mathsf{m}\right)^{\frac{1}{2}} 
	\longrightarrow ~0 \quad \text{as} \quad k \rightarrow \infty.
	\end{align*}
	Letting $k \rightarrow \infty$ in \eqref{utolso_egyenlotlenseg} and using the fact that $r_F < \infty$ completes the proof.
\end{proof}

If we apply Theorem \ref{theorem_inclusion} to a compact set $K = M \setminus \Omega$ of zero capacity, where $\Omega \subset M$ is an open set, 
we obtain the inclusion \eqref{inclusion1}. More precisely, the following result holds.

\begin{cor} \label{corollForInclusion}
	Let $(M,F)$ be a geodesically complete, $2$-hyperbolic Finsler manifold such that $F$ is uniformly convex, $r_F < \infty$ and 
	$\mathrm{Ric}_N \geq -\kappa$, $\kappa > 0$. 
	Let $\Omega \subset M$ be an open set such that $M \setminus \Omega$ is compact with $\capac_2(M \setminus \Omega, M) = 0$. 
	Then we have the following inclusion:
	\begin{equation*}
	D^{1,2}(M) \subset D^{1,2}(\Omega).
	\end{equation*}
\end{cor}

Applying Corollary \ref{corollForInclusion} yields Corollary \ref{corollForInclusion2}.

\noindent
\textbf{Acknowledgment.}
Á. Mester is supported by the National Research, Development and Innovation Fund of Hungary, financed under the K$\_$18 funding scheme, Project No. 127926. 
An anonymous reviewer is thanked for carefully reading the manuscript and suggesting substantial improvements.

\end{document}